\documentclass[12pt]{article}

\title{Mean action and the Calabi invariant}
\author{Michael Hutchings\footnote{Partially supported by NSF grant DMS-1406312.}}
\date{}

\usepackage{amssymb}
\usepackage{latexsym}
\usepackage{amsmath}
\usepackage{amsthm}
\usepackage{amscd}
\usepackage[dvips]{graphics}
\usepackage{overpic}
\usepackage[mathscr]{euscript}

\newcommand{\mc}[1]{{\mathcal #1}}

\numberwithin{equation}{section}

\newtheorem{theorem}{Theorem}[section]
\newtheorem{proposition}[theorem]{Proposition}
\newtheorem{corollary}[theorem]{Corollary}
\newtheorem{lemma}[theorem]{Lemma}
\newtheorem{lemma-definition}[theorem]{Lemma-Definition}

\theoremstyle{definition}

\newtheorem{remark}[theorem]{Remark}

\newtheorem{example}[theorem]{Example}

\newcommand{\floor}[1]{\left\lfloor #1 \right\rfloor}
\newcommand{\ceil}[1]{\left\lceil #1 \right\rceil}

\newcommand{\C}{{\mathbb C}}
\newcommand{\Q}{{\mathbb Q}}
\newcommand{\R}{{\mathbb R}}
\newcommand{\N}{{\mathbb N}}
\newcommand{\Z}{{\mathbb Z}}

\newcommand{\A}{{\mathcal A}}
\newcommand{\op}{\operatorname}

\newcommand{\M}{\mc{M}}

\newcommand{\Ker}{\op{Ker}}

\newcommand{\current}{\mathscr{C}}
\newcommand{\calabi}{\mathcal{V}}

\newcommand{\bpm}{\begin{pmatrix}}
\newcommand{\epm}{\end{pmatrix}}

\renewcommand{\epsilon}{\varepsilon}

\begin{document}

\setcounter{tocdepth}{2}

\maketitle

\begin{abstract}
Given an area-preserving diffeomorphism of the closed unit disk which is a rotation near the boundary, one can naturally define an ``action'' function on the disk which agrees with the rotation number on the boundary. The Calabi invariant of the diffeomorphism is the average of the action function over the disk. Given a periodic orbit of the diffeomorphism, its ``mean action'' is defined to be the average of the action function over the orbit. We show that if the Calabi invariant is less than the boundary rotation number, then the infimum over periodic orbits of the mean action is less than or equal to the Calabi invariant. The proof uses a new filtration on embedded contact homology determined by a transverse knot, which may be of independent interest. (An analogue of this filtration can be defined for any other version of contact homology in three dimensions that counts holomorphic curves.)
\end{abstract}

\section{Introduction}

There is a close relation between the dynamics of Reeb flows on $S^3$ and the dynamics of area-preserving diffeomorphisms of surfaces. This relation was developed by Hofer-Wysock-Zehnder \cite{hwzannals1}, who showed that ``dynamically convex'' Reeb flows on $S^3$ admit disk-type global surfaces of section. Combining this fact with a result of Franks \cite{franks} on area-preserving homeomorphisms of the annulus, they showed that the Reeb flow on every strictly convex hypersurface in $\R^4$ has either two or infinitely many periodic orbits. This result was extended in \cite{hwzannals2} to generic star-shaped hypersurfaces in $\R^4$. More recently, Abbondandolo-Bramham-Hryniewicz-Salom\~ao \cite[Thm.\ 2]{long} constructed counterexamples to a conjecture about the systolic ratios of contact forms on $S^3$, by finding certain examples of area-preserving diffeomorphisms of the disk, and then suspending them to obtain Reeb flows on $S^3$.

The works mentioned above use facts about area-preserving maps of surfaces to obtain results about Reeb flows on $S^3$. In this paper we will proceed in the opposite direction, using information about Reeb flows on $S^3$, coming from embedded contact homology, to obtain results about area-preserving diffeomorphisms of the disk.

\subsection{Statement of the main result}

Let $D^2$ denote the closed unit disk, and let $\omega$ denote the standard area form on $D^2$, rescaled so that its total area is $1$. Let 
\[
\phi:(D^2,\omega)\longrightarrow (D^2,\omega)
\]
be an area-preserving diffeomorphism. Assume that there exists a real number $\theta_0$ so that near the boundary of $D^2$, the map $\phi$ is rotation by angle $2\pi\theta_0$. That is, in standard polar coordinates $(r,\theta)$, we have
\begin{equation}
\label{eqn:rotation}
\phi(r,\theta) = (r,\theta+2\pi\theta_0)
\end{equation}
when $r$ is sufficiently close to $1$. Let $G$ denote the set of pairs $(\phi,\theta_0)$ as above.

Let $\beta$ be a primitive of $\omega$ on $D^2$. Assume that on the boundary of $D^2$ we have
\begin{equation}
\label{eqn:betaboundary}
\beta|_{\partial D^2}(\partial_\theta) = \frac{1}{2\pi}.
\end{equation}
Given $(\phi,\theta_0)\in G$, we define the ``action function'' $f:D^2\to\R$ to be the unique function such that
\begin{equation}
\label{eqn:f}
\begin{split}
df &= \phi^*\beta-\beta,\\
f|_{\partial D^2} &= \theta_0.
\end{split}
\end{equation}
Although the function $f$ depends on $\beta$, it determines two important quantities which do not.

First, the {\em Calabi invariant\/} \cite{calabi,gg} is defined by\footnote{We use the letter $\calabi$ to denote the Calabi invariant because of the connection with contact volume in Proposition~\ref{prop:openbook} below.}
\[
\calabi(\phi,\theta_0) = \int_{D^2}f\omega.
\]
One can think of the Calabi invariant as a kind of ``rotation number'' of the map $\phi$. For example, if $\phi$ is rotation by angle $2\pi\theta_0$, i.e.\ if equation \eqref{eqn:rotation} holds for all $(r,\theta)$, then $\calabi(\phi,\theta_0) = \theta_0$. In general, the Calabi invariant is a homomorphism $G\to\R$, see e.g.\ \cite{gg}.

Second, if $\gamma=(x_1,\ldots,x_d)$ is a periodic orbit of $\phi$ of period $d$, that is if $\phi(x_i)=x_{i+1\mod d}$, and the points $x_i$ are distinct\footnote{We assume that the points $x_i$ are distinct only to avoid confusion with the notation later.}, we define the {\em total action\/}
\[
\A(\gamma) = \sum_{i=1}^df(x_i).
\]

\begin{lemma}
Fix $(\phi,\theta_0)\in G$. Then the total action $\A(\gamma)$ of a periodic orbit $\gamma$, and the Calabi invariant
$\calabi(\phi,\theta_0)$, do not depend on the primitive $\beta$ of $\omega$ satisfying \eqref{eqn:betaboundary}.
\end{lemma}

\begin{proof}
Let $f_\beta$ denote the action function determined by $\beta$.
If $x\in D^2$, we have
\begin{equation}
\label{eqn:fbeta}
f_\beta(x) = \theta_0 + \int_\eta (\phi^*\beta - \beta)
\end{equation}
where $\eta$ is any path in $D^2$ from $\partial D^2$ to $x$.

If $\gamma=(x_1,\ldots,x_d)$ is a periodic orbit, choose paths $\eta_i$ from $\partial D^2$ to $x_i$ so that $\eta_{i+1}=\phi_*\eta_i$ for $i=1,\ldots,d-1$. Then by \eqref{eqn:fbeta}, the total action of $\gamma$ determined by $\beta$ is
\[
\sum_{i=1}^df_\beta(x_i) = d\theta_0 + \int_{\phi_*\eta_d - \eta_1}\beta.
\]
This does not depend on $\beta$, because $\phi_*\eta_d-\eta_1$ is homologous to a path on the boundary of $D^2$, along which the integral of $\beta$ is determined by \eqref{eqn:betaboundary}.

Similar calculations show that $\calabi(\phi,\theta_0)$ does not depend on $\beta$, see e.g.\ \cite{gg}.
\end{proof}

Let $\mc{P}(\phi)$ denote the set of periodic orbits of $\phi$. If $\gamma=(x_1,\ldots,x_d)\in\mc{P}(\phi)$ is a periodic orbit, denote its period by $d(\gamma)=d$, and define its {\em mean action\/} to be
\[
\frac{\A(\gamma)}{d(\gamma)} = \frac{1}{d} \sum_{i=1}^df(x_i).
\]
Our main result is the following:

\begin{theorem}
\label{thm:main}
Let $\theta_0\in\R$, and let $\phi$ be an area-preserving diffeomorphism of $D^2$ which agrees with rotation by angle $2\pi\theta_0$ near the boundary. Suppose that
\begin{equation}
\label{eqn:mainass}
\calabi(\phi,\theta_0) < \theta_0.
\end{equation}
Then
\begin{equation}
\label{eqn:conclusion}
\inf\left\{\frac{\A(\gamma)}{d(\gamma)} \;\bigg|\; \gamma\in\mc{P}(\phi)\right\} \le \calabi(\phi,\theta_0).
\end{equation}
\end{theorem}

The conclusion \eqref{eqn:conclusion} can be restated as follows: for every $\epsilon>0$, there exists a periodic orbit $\gamma$ such that the average of $f$ over $\gamma$ is less than $\epsilon$ plus the average of $f$ over $D^2$.

\begin{remark}
The conclusion of Theorem~\ref{thm:main} does not necessarily hold if one takes the infimum over fixed points instead of periodic orbits. One can make a simple counterexample as follows, cf.\ \cite[\S2]{long}: Start with a rotation by angle $\pi$, with $\theta_0=1/2$. This has Calabi invariant $1/2$, and the only fixed point is the origin, which has action $1/2$. Now compose with a negative twist supported in a subdisk contained in the interior of one half of $D^2$; this will decrease the Calabi invariant without creating any new fixed points.
\end{remark}

\begin{remark}
We do not know how to prove\footnote{It is natural to try to prove this using arguments similar to those in \S\ref{sec:input}. To start, there is a straightforward modification of Proposition~\ref{prop:input} in which the inequalities are essentially switched. However the calculations at the end of \S\ref{sec:input} then do not work out.} the reverse inequality in Theorem~\ref{thm:main}, namely that the supremum of the mean action is greater than or equal to the Calabi invariant assuming \eqref{eqn:mainass}.  However the reverse inequality is arguably less interesting; for example it trivially holds when $\theta_0$ is rational, as then every point near the boundary of $D^2$ is on a periodic orbit with mean action equal to $\theta_0$.
\end{remark}

Replacing $(\phi,\theta_0)$ by $(\phi^{-1},-\theta_0)$ in Theorem~\ref{thm:main} immediately gives:

\begin{corollary}
Let $\theta_0\in\R$, and let $\phi$ be an area-preserving diffeomorphism of $D^2$ which agrees with rotation by angle $2\pi\theta_0$ near the boundary. Suppose that
\[
\calabi(\phi,\theta_0) > \theta_0.
\]
Then
\[
\sup\left\{\frac{\A(\gamma)}{d(\gamma)} \;\bigg|\; \gamma\in\mc{P}(\phi)\right\} \ge \calabi(\phi,\theta_0).
\]
\end{corollary}

\paragraph{Outline.} The rest of this paper is arranged in order of increasing mathematical prerequisites. In \S\ref{sec:reduction} we use elementary constructions to reduce Theorem~\ref{thm:main} to a statement about Reeb flows on $S^3$ (Proposition~\ref{prop:reduced}). In \S\ref{sec:input} we prove Proposition~\ref{prop:reduced} using input from embedded contact homology (Proposition~\ref{prop:input}). This input has two parts which, roughly speaking, give an upper bound on $\A(\gamma)$ and a lower bound on $d(\gamma)$ in \eqref{eqn:conclusion}. The upper bound on action comes from a relation between the ECH spectrum and contact volume\footnote{Egor Shelukhin has suggested to the author that to prove Theorem~\ref{thm:main}, instead of using the relation between volume and the asymptotics of ECH capacities for contact forms on $S^3$, one might alternatively use the relation between the Calabi invariant and the Entov-Polterovich quasimorphism \cite[Prop.\ 3.3]{ep} for Hamiltonian diffeomorphisms of $S^2$.} which is reviewed in \S\ref{sec:review}. The lower bound on the period comes from a new filtration on embedded contact homology determined by a transverse knot, which is introduced in \S\ref{sec:filtration}. Proposition~\ref{prop:input} is then proved at the end of \S\ref{sec:filtration}. This proof uses topological invariance of the knot filtration, which is proved in \S\ref{sec:functoriality}, following some technical preliminaries on ECH cobordism maps in \S\ref{sec:cobordism}.

\paragraph{Acknowledgments.} We thank Umberto Hryniewicz for suggesting that something like Theorem~\ref{thm:main} should be true, and Vinicius Ramos for related helpful discussions.

\section{Reduction to contact geometry}
\label{sec:reduction}

To prove Theorem~\ref{thm:main}, we will reduce to a statement about contact geometry. 

\subsection{Preliminary definitions}

The contact geometry statement that we will need uses the following definitions, all of which are standard except possibly for the ``rotation number'' \eqref{eqn:rn} below.

Let $Y$ be a closed oriented three-manifold. Recall that a {\em contact form\/} on $Y$ is a $1$-form $\lambda$ such that $\lambda\wedge d\lambda>0$ everywhere. We define the {\em volume\/}
\[
\op{vol}(Y,\lambda) = \int_Y\lambda\wedge d\lambda.
\]

A contact form $\lambda$ determines the {\em Reeb vector field\/} $R$ characterized by $d\lambda(R,\cdot)=0$ and $\lambda(R)=1$, and the {\em contact structure\/} $\xi=\Ker(\lambda)\subset TY$.

A {\em Reeb orbit\/} is a smooth map $\gamma:\R/T\Z\to Y$ for some $T>0$, modulo reparametrization, such that $\gamma'(t)=R(\gamma(t))$. The {\em symplectic action\/} of $\gamma$ is defined by
\[
\A(\gamma) = \int_\gamma\lambda.
\]
This is the same as the above period $T$. We say that the Reeb orbit $\gamma$ is {\em simple\/} if the map $\gamma$ is an embedding. We let $\mc{P}(\gamma)$ denote the set of simple Reeb orbits of $\gamma$.

If $\gamma$ is a Reeb orbit, then the linearization of the time $T$ Reeb flow along $\gamma$ determines a symplectic linear map
\begin{equation}
\label{eqn:lrm}
P_\gamma: (\xi_{\gamma(0)},d\lambda) \to (\xi_{\gamma(0)},d\lambda).
\end{equation}
We say that $\gamma$ is {\em elliptic\/} if $P_\gamma$ is a rotation for some symplectic identification $\xi_{\gamma(0)}\simeq \R^2$. (When the eigenvalues of $P_\gamma$ are not both $1$ or $-1$, an equivalent condition is that the eigenvalues of $P_\gamma$ are on the unit circle.)

If $H_1(Y)=0$, then every Reeb orbit $\gamma$ has a well-defined {\em rotation number\/} $\op{rot}(\gamma)\in\R$ defined as follows. Consider a symplectic trivialization $\tau: \xi|_\gamma \simeq \gamma\times \R^2$ such that a pushoff of $\gamma$ via this trivialization has linking number zero\footnote{When defining the Conley-Zehnder index of a nondegenerate Reeb orbit, it is common to use a different homotopy class of trivialization $\tau$, chosen so that it extends to a trivialization of $\xi$ over a surface bounded by $\gamma$. These two trivializations differ by the self-linking number of the transverse knot $\gamma$, see \S\ref{sec:review}.} with $\gamma$. If $\gamma$ is elliptic, then we can choose the trivialization $\tau$ so that for each $t\in[0,T]$, the linearization of the time $t$ Reeb flow $\xi_{\gamma(0)}\to \xi_{\gamma(t)}$ is rotation by angle $2\pi\theta_t$, where $\theta:[0,T]\to\R$ is continuous and $\theta_0=0$. We then define
\begin{equation}
\label{eqn:rn}
\op{rot}(\gamma) = \theta_T.
\end{equation}
If $\gamma$ is not elliptic, then $P_\gamma$ has real eigenvalues, possibly both equal to $\pm1$, which are both positive or both negative. In this case we define $\op{rot}(\gamma)$ to be the number of times that the linearized Reeb flow rotates an eigenvector of $P_\gamma$ with respect to the trivialization $\tau$ as one flows around $\gamma$. This is an integer when $P_\gamma$ has positive eigenvalues, and a half-integer when $P_\gamma$ has negative eigenvalues.

Recall that an {\em open book decomposition\/} of a closed oriented three-manifold $Y$ consists of a compact oriented surface $\Sigma$ with boundary, a diffeomorphism $\psi:\Sigma\to\Sigma$ which is the identity near the boundary, and an orientation-preserving diffeomorphism
\[
Y \simeq \frac{[0,1]\times \Sigma}{\sim}
\]
where on the right hand side we identify
\[
(1,x)\sim(0,\psi(x)), \quad x\in\Sigma
\]
and
\[
(\tau,x)\sim (\tau',x),\quad x\in\partial\Sigma,\;\tau,\tau'\in[0,1].
\] 
The image in $Y$ of a set $\{\tau\}\times \op{int}(\Sigma)$ is called a {\em page\/} of the open book. The image in $Y$ of the set $[0,1]\times\partial\Sigma$ is called the {\em binding\/} of the open book. A contact form $\lambda$ is {\em compatible\/} with an open book decomposition of $Y$ if the Reeb vector field is positively transverse to the pages, and tangent to the binding. Then each component of the binding is a (simple) Reeb orbit.

\subsection{From a disk map to an open book}

We can now state the following proposition, which will allow us to reduce Theorem~\ref{thm:main} to a statement about contact forms on $S^3$. (This is similar to \cite[Prop.\ 3.2]{long}. See \cite[\S6]{cgh} for some related constructions.) Below, we fix the primitive $\beta$ of $\omega$ to be
\[
\beta = \frac{r^2}{2\pi}d\theta.
\]

\begin{proposition}
\label{prop:openbook}
Let $\theta_0\in\R$, and let $\phi$ be an area-preserving diffeomorphism of $D^2$ which is rotation by angle $2\pi\theta_0$ near the boundary. Suppose that the action function $f$ defined in \eqref{eqn:f} is everywhere positive. Then there exists a contact form $\lambda$ on $S^3$, which is compatible with an open book decomposition in which the pages are disks, with the following properties:
\begin{description}
\item{(a)}
The binding orbit $B$ has action $1$ and is elliptic with rotation number $\theta_0^{-1}$.
\item{(b)}
There is a bijection $\mc{P}(\gamma) = \{B\}\cup \mc{P}(\phi)$, such that if $\gamma\in\mc{P}(\phi)$, then the corresponding Reeb orbit of $\lambda$ has linking number $d(\gamma)$ with the binding, and symplectic action equal to the total action of $\gamma$.
\item{(c)}
$\op{vol}(S^3,\lambda) = \calabi(\phi,\theta_0)$.
\end{description}
\end{proposition}

\begin{proof}
Let $Y$ denote the mapping torus of $\phi$, namely
\[
\begin{split}
Y &= [0,1]\times D^2/\sim,\\
(1,x) &\sim (0,\phi(x)).
\end{split}
\]
To prove the proposition we proceed in three steps.

{\em Step 1.\/} We first construct a $1$-form $\lambda_0$ on $[0,1]\times D^2$ with the following properties:
\begin{description}
\item{(i)} Near $[0,1]\times\partial D^2$ we have
\begin{equation}
\label{eqn:lb}
\lambda_0 = \theta_0dt+\frac{r^2}{2\pi}d\theta
\end{equation}
where $t$ denotes the $[0,1]$ coordinate.
\item{(ii)} $\lambda_0$ descends to a smooth $1$-form on the mapping torus $Y$.
\item{(iii)} $\lambda_0$ is a contact form.
\item{(iv)} The Reeb vector field of $\lambda_0$ is a positive multiple of $\partial_t$.
\item{(v)} The time for the Reeb flow to move from $(0,x)$ to $(1,x)$ equals $f(x)$.
\item{(vi)} $\op{vol}([0,1]\times D^2,\lambda_0) = \calabi(\phi,\theta_0)$.
\end{description}

To construct $\lambda_0$ with the above properties, choose $\epsilon\in(0,\min(f)/\max(f))$. Let $\eta:[0,1]\to\R$ be a smooth function with the following properties:
\begin{itemize}
\item $\eta(t)=0$ for $t$ close to $0$.
\item $\eta(t)=t-1$ for $t$ close to $1$.
\item $-\epsilon<\eta'(t)\le 1$ for all $t$.
\end{itemize}
Now define
\begin{equation}
\label{eqn:fdtb}
\lambda_0 = Fdt + \beta(t)
\end{equation}
where $F:[0,1]\times D^2\to \R$ is defined by
\[
F(t,x) = (1-\eta'(t))f(x) + \eta'(t)\phi^*f(x),
\]
and $\beta(t)$ is a $1$-form on $D^2$ for each $t\in[0,1]$ defined by
\[
\beta(t) =  \beta + (t-\eta(t))df + \eta(t)\phi^*df.
\]

Property (i) is immediate from the above definition and the fact that $f\equiv\phi^*f\equiv\theta_0$ near $\partial D^2$.

The conditions on $\eta$ imply that
\[
\lambda_0 = fdt + \beta + t(\phi^*\beta - \beta)
\]
for $t$ close to $0$, and
\[
\lambda_0 = \phi^*fdt + \phi^*\beta + (t-1)\phi^*(\phi^*\beta -\beta)
\]
for $t$ close to $1$. This implies property (ii).

To prove property (iii), we observe that
\[
\frac{d\beta(t)}{dt} = dF(t,\cdot),
\]
and therefore
\begin{equation}
\label{eqn:dlo}
d\lambda_0 = d\beta(t)=\omega.
\end{equation}
Thus
\begin{equation}
\label{eqn:fdto}
\lambda_0\wedge d\lambda_0 = Fdt\wedge \omega.
\end{equation}
Our assumptions on $\eta$ and $\epsilon$ imply that $F$ is everywhere positive, so $\lambda_0$ is a contact form.

It follows from \eqref{eqn:fdtb} and \eqref{eqn:dlo} that the Reeb vector field associated to $\lambda_0$ is given by
\[
R_0=F^{-1}\partial_t.
\]
The time to flow from $(0,x)$ to $(1,x)$ is then
\begin{equation}
\label{eqn:iff}
\int_{0}^1F(t,x)dt = f(x).
\end{equation}
This proves properties (iv) and (v).

Property (vi) follows from equations \eqref{eqn:fdto} and \eqref{eqn:iff}.

{\em Step 2.\/} We now define a contact form $\lambda$ on $S^3$ such that:
\begin{itemize}
\item $\lambda$ is compatible with an open book decomposition in which the pages are disks.
\item The complement of the binding can be identified with the interior of the mapping torus $Y$ so that the pages are identified with sets $\{t\}\times\op{int}(D^2)$, and $\lambda$ is identified with $\lambda_0$.
\end{itemize}

Choose $\epsilon>0$ sufficiently small so that $\phi$ is rotation by angle $\phi_0$ in the region where $r^2\ge 1-\epsilon^2$. Let $D^2(\epsilon)$ denote the disk of radius $\epsilon$.
Let $Z=S^1\times D^2(\epsilon)$. We denote the $S^1$ coordinate on $Z$ by $\tau\in\R/\Z$, and on the $D^2(\epsilon)$ factor we use polar coordinates $(\rho,\psi)$ where $\rho\in[0,\epsilon]$ and $\psi\in\R/2\pi\Z$.

Let
\[
M = \frac{\op{int}(Y) \sqcup Z}{\sim}
\]
where the equivalence relation $\sim$ identifies the subset of $Z$ where $\rho>0$, with the subset of $\op{int}(Y)$ where $r^2\ge 1-\epsilon^2$, by the equations
\begin{equation}
\label{eqn:ident}
\begin{split}
r^2+\rho^2 &= 1,\\
\psi &= 2\pi t,\\
\theta &= 2\pi\tau - \theta_0\psi,\quad\quad 0\le \psi < 2\pi.
\end{split}
\end{equation}
Observe that $M$ is diffeomorphic to $S^3$ and has an open book decomposition in which the pages are the disks $\{t\}\times \op{int}(D^2)$ in $Y$, and the binding is the circle $S^1\times\{0\}$ in $Z$.

Now we claim that the contact form $\lambda_0$ on $Y$ extends to a smooth contact form $\lambda$ on $M$. To see this, on the subset of $Z$ where $\rho>0$, by equations \eqref{eqn:lb} and \eqref{eqn:ident} we can rewrite $\lambda_0$ in the coordinates $\tau,\rho,\psi$ as
\[
\lambda_0 = \frac{\theta_0}{2\pi}\rho^2d\psi + (1-\rho^2)d\tau.
\]
To see that this extends smoothly over the circle $\rho=0$, introduce new coordinates $x,y$ on $Z$ defined by $x=\rho\cos\psi$ and $y=\rho\sin\psi$. Then
\[
\lambda_0 = \frac{\theta_0}{2\pi}(xdy-ydx) + (1-x^2-y^2)d\tau.
\]
This extends smoothly over the circle $\rho=0$, and we denote the extension by $\lambda$.
We then compute that
\[
\lambda\wedge d\lambda = d\tau dx dy,
\]
so $\lambda$ is a contact form. 

We now check that $\lambda$ is compatible with the open book decomposition. It follows from part (iv) of Step 1 that the Reeb vector field $R$ of $\lambda$ is positively transverse to the pages. To see that the binding is a Reeb orbit, we compute that on $Z$ the Reeb vector field is given by
\begin{equation}
\label{eqn:bz}
R = 2\pi\theta_0^{-1}\partial_\psi + \partial_\tau
\end{equation}
which agrees with $\partial_\tau$ on the binding.

{\em Step 3.\/} We now show that the contact form $\lambda$ on $M\simeq S^3$ satisfies the conclusions (a)--(c) of the proposition.

Assertion (a) follows from equation \eqref{eqn:bz}.

The bijection $\mc{P}(\gamma) = \{B\}\cup \mc{P}(\phi)$ relating linking number to period follows from part (iv) of Step 1. Part (v) of Step 1 implies that symplectic action corresponds to total action under this bijection. This proves assertion (b).

Assertion (c) follows from part (vi) of Step 1.
\end{proof}

\subsection{Proof of the main theorem using contact geometry}
\label{sec:twist}

We now use Proposition~\ref{prop:openbook} to reduce Theorem~\ref{thm:main} to the following:

\begin{proposition}
\label{prop:reduced}
Let $\lambda$ be a contact form on $S^3$ with volume $V$. Suppose that $\lambda$ is compatible with an open book decomposition in which the pages are disks, and the binding $B$ has symplectic action $1$. Suppose that the binding is elliptic with irrational rotation number $\theta_0^{-1}$, and suppose that
\[
V < \theta_0.
\]
Then
\[
\inf\left\{\frac{\A(\gamma)}{\ell(\gamma,B)}\;\bigg|\; \gamma\in\mc{P}(\gamma)\setminus\{B\}\right\} \le \sqrt{\theta_0V}.
\]
\end{proposition}

Here $\ell(\gamma,B)$ denotes the linking number of the simple Reeb orbit $\gamma$ with the binding $B$.

\begin{proof}[Proof of Theorem~\ref{thm:main} (assuming Proposition~\ref{prop:reduced}).]
We can assume without loss of generality that $f$ is everywhere positive, because given an arbitrary $(\phi,\theta_0)\in G$, we can make $f$ positive by adding a large positive integer $n$ to $\theta_0$. This will not affect the validity of Theorem~\ref{thm:main} since both the Calabi invariant and the mean action of every periodic orbit will be increased by $n$.

We now proceed in three steps.

{\em Step 1.\/} We first restrict to the case where $\theta_0$ is irrational and prove the weaker inequality
\begin{equation}
\label{eqn:weaker}
\op{inf}\left\{\frac{\A(\gamma)}{d(\gamma)} \;\bigg|\; \gamma\in\mc{P}(\phi)\right\} \le \sqrt{\theta_0\cdot\calabi(\phi,\theta_0)}.
\end{equation}
Since $f>0$, we obtain a contact form $\lambda$ on $S^3$ from Proposition~\ref{prop:openbook}. Applying Proposition~\ref{prop:reduced} to this contact form immediately gives \eqref{eqn:weaker}.

{\em Step 2.\/} Continuing to restrict to the case where $\theta_0$ is irrational, we now deduce \eqref{eqn:conclusion} from the weaker result of Step 1 by twisting the map $\phi$ near the boundary to make the boundary rotation number closer to the Calabi invariant. 

To start, choose $\delta>0$ so that $\phi$ is rotation by $2\pi\theta_0$ in the region where $r\ge 1-\delta$. Write $V=\calabi(\phi,\theta_0)$ and pick $\epsilon>0$ with $V+\epsilon<\theta_0$.  Choose a smooth nonincreasing function
\[
\psi:[1-\delta,1] \longrightarrow [V+\epsilon,\theta_0]
\]
such that $\psi(r)=\theta_0$ for $r$ close to $1-\delta$ and $\psi(r)=V+\epsilon$ for $r$ close to $1$. Define a new area-preserving diffeomorphism
\[
\widehat{\phi}:D^2\longrightarrow D^2
\]
in polar coordinates by
\[
\widehat{\phi}(r,\theta) = \left\{\begin{array}{cl} \phi(r,\theta), & r\le 1-\delta,\\
(r,\theta+2\pi\psi(r)), & r\ge 1-\delta.\end{array}\right.
\]
We want to apply Step 1 to the pair $(\widehat{\phi},V+\epsilon)\in G$.

To see what Step 1 gives us, we need to compute the action function
\[
\widehat{f}: D^2\to \R
\]
determined by $(\widehat{\phi},V+\epsilon)$.
If $1-\delta\le r\le 1$ then $\widehat{f}$ is a function of $r$, namely
\begin{equation}
\label{eqn:fhatout}
\widehat{f}(r,\theta) =
V+\epsilon - \int_r^1\rho^2\psi'(\rho)d\rho.
\end{equation}
If $r\le 1-\delta$, then $\widehat{f}$ differs from the old action function $f$ by a constant:
\begin{equation}
\label{eqn:fhatin}
\widehat{f}(r,\theta) = f(r,\theta) + \left[ V + \epsilon - \theta_0 - \int_{1-\delta}^1\rho^2\psi'(\rho)d\rho \right].
\end{equation}
It follows from \eqref{eqn:fhatout} and \eqref{eqn:fhatin} that if we choose $\delta$ sufficiently small, then
\begin{equation}
\label{eqn:fhatorbit}
\left|\hat{f}(r,\theta)-f(r,\theta)\right| < \epsilon/3
\end{equation}
for $r\le 1-\delta$, and 
\begin{equation}
\label{eqn:Vhat}
\left|\calabi(\widehat{\phi},V+\epsilon)-V\right| < \epsilon/2.
\end{equation}

By \eqref{eqn:Vhat}, Step 1 is applicable to the pair $\left(\widehat{\phi},V+\epsilon\right)$, and we obtain
\begin{equation}
\label{eqn:weakerapplied}
\op{inf}\left\{\frac{\widehat{\A}(\gamma)}{d(\gamma)} \;\bigg|\; \gamma\in\mc{P}\left(\widehat{\phi}\right)\right\} \le \sqrt{(V+\epsilon)(V+\epsilon/2)}
\end{equation}
where $\widehat{\A}$ denotes the total action defined using the action function $\widehat{f}$. Now the inequality \eqref{eqn:weakerapplied} is still true if in the infimum on the left hand side we restrict to periodic orbits that are contained in the region $r\le 1-\delta$. The reason is that by \eqref{eqn:fhatout}, any periodic orbit in the region where $r\ge 1-\delta$ has mean action at least $V+\epsilon$, which is greater than the right hand side of \eqref{eqn:weakerapplied}. Thus we have
\[
\op{inf}\left\{\frac{\widehat{\A}(\gamma)}{d(\gamma)} \;\bigg|\; \gamma\in\mc{P}\left(\widehat{\phi}\big|_{(r\le 1-\delta)}\right)\right\} \le \sqrt{(V+\epsilon)(V+\epsilon/2)}.
\]
On the other hand, $\widehat{\phi}$ agrees with $\phi$ on the region where $r\le 1-\delta$, and if $\gamma$ is a periodic orbit in this region, then by \eqref{eqn:fhatin}, the total actions determined by $\widehat{f}$ and $f$ satisfy
\[
\left|\widehat{\A}(\gamma) - \A(\gamma)\right| < d(\gamma)\epsilon/3.
\]
Thus we obtain
\[
\inf\left\{\frac{\A(\gamma)}{d(\gamma)} \;\bigg|\; \gamma\in\mc{P}(\phi)\right\} \le \sqrt{(V+\epsilon)(V+\epsilon/2)} + \epsilon/3.
\]
Since $\epsilon>0$ can be arbitrarily small, this gives
\[
\inf\left\{\frac{\A(\gamma)}{d(\gamma)} \;\bigg|\; \gamma\in\mc{P}(\phi)\right\} \le V
\]
as desired.

{\em Step 3.\/} We now drop the assumption that $\theta_0$ is irrational. To do so, pick $\epsilon>0$ small. As in Step 2, we can twist near the boundary to replace $(\phi,\theta_0)$ by $\left(\widehat{\phi},\widehat{\theta_0}\right)$ where $\widehat{\theta_0}$ is irrational and close to $\theta_0$, and the inequalities \eqref{eqn:fhatorbit} and \eqref{eqn:Vhat} hold. Then applying the result of Step 2 to the pair $\left(\widehat{\phi},\widehat{\theta_0}\right)$, we find that
\[
\inf\left\{\frac{A(\gamma)}{d(\gamma)} \;\bigg|\; \gamma\in\mc{P}(\phi)\right\} \le V + 5\epsilon/6.
\]
Since $\epsilon>0$ was arbitrary, the result \eqref{eqn:conclusion} follows.
\end{proof}

\section{Input from embedded contact homology}
\label{sec:input}

The proof of Proposition~\ref{prop:reduced} uses Proposition~\ref{prop:input} below, which is proved using embedded contact homology. To state the latter result, we need the following preliminary definitions.

Let $\lambda$ be a contact form on a closed oriented three-manifold $Y$. An {\em orbit set\/} is a finite set of pairs $\alpha=\{(\alpha_i,m_i)\}$ where the $\alpha_i$ are distinct simple Reeb orbits and the $m_i$ are positive integers. One can think of the integers $m_i$ as ``multiplicities''. We sometimes write an orbit set using the multiplicative notation
\[
\alpha = \prod_i\alpha_i^{m_i}.
\]
Let $\mc{O}(\lambda)$ denote the set of orbit sets. The symplectic action of an orbit set is defined by
\begin{equation}
\label{eqn:aalpha}
\A(\alpha) = \sum_im_i\A(\alpha_i).
\end{equation}
If $H_1(Y)=0$ and if $K$ is an oriented knot disjoint from the Reeb orbits $\alpha_i$, then we define the linking number
\begin{equation}
\label{eqn:lalpha}
\ell(\alpha,K) = \sum_im_i\ell(\alpha_i,K).
\end{equation}

If $a$ and $b$ are positive real numbers, define $\{N_k(a,b)\}_{k=0,1,\ldots}$ to be the sequence of nonnegative integer linear combinations of $a$ and $b$, written in nondecreasing order with repetitions. We can now state:

\begin{proposition}
\label{prop:input}
Let $\lambda$ be a contact form on $S^3$. Suppose that $\lambda$ is compatible with an open book decomposition in which the pages are disks, and the binding $B$ is elliptic with positive irrational rotation number. Then for every $\epsilon>0$, if $k$ is a sufficiently large positive integer, then there exists an orbit set $\alpha$ not including the binding, and a nonnegative integer $m$, such that
\begin{align}
\label{eqn:abound}
\frac{(\A(\alpha)+m\A(B))^2}{2k} &\le \op{vol}(S^3,\lambda) + \epsilon,\\
\label{eqn:lbound}
\ell(\alpha,B) + m\op{rot}(B) &\ge N_k(1,\op{rot}(B)).
\end{align}
\end{proposition}

Proposition~\ref{prop:input} will be proved in \S\ref{sec:filtration}. Assuming this proposition, we now prove Proposition~\ref{prop:reduced}.

We first need the following combinatorial lemma. We give the elementary proof here because we will need equation \eqref{eqn:latticepoints} again in Lemma~\ref{lem:ellipsoid}.

\begin{lemma}
\label{lem:comb}
Given positive real numbers $a,b$, there exists a constant $c$ such that
\begin{equation}
\label{eqn:comb}
N_k(a,b)^2\ge 2abk-ck^{1/2}
\end{equation}
for every nonnegative integer $k$.
\end{lemma}

\begin{proof}
We have $N_k(a,b)=am+bn$ for some nonnegative integers $m,n$. Write $L=am+bn$. Let $T$ be the triangle in the plane bounded by the axes and the line $ax+by=L$. Then by the definition of $N_k$, the number of lattice points in the triangle $T$ (including the boundary) is at least $k+1$ (with equality when $(m,n)$ is the only lattice point on the diagonal edge). By dividing the triangle $T$ into the rectangle with corners at $(0,0),(m,0),(0,n),(m,n)$ and two smaller triangles, we find that the number of lattice points in $T$ is
\begin{align}
\label{eqn:latticepoints}
|T\cap \Z^2| &= (m+1)(n+1) + \sum_{i=1}^m\floor{\frac{ai}{b}} + \sum_{j=1}^n\floor{\frac{bj}{a}}\\
\nonumber
&\le (m+1)(n+1) + \sum_{i=1}^m\frac{ai}{b} + \sum_{j=1}^n\frac{bj}{a}\\
\nonumber
&= (m+1)(n+1) + \frac{am(m+1)}{2b} + \frac{bn(n+1)}{2a}\\
\nonumber
&= \frac{L^2}{2ab} + \frac{L}{2a} + \frac{L}{2b} + \frac{m+n}{2} + 1.
\end{align}
Thus
\[
2abk \le L^2 + (a+b)L + ab(m+n).
\]
The estimate \eqref{eqn:comb} follows from this.
\end{proof}

\begin{proof}[Proof of Proposition~\ref{prop:reduced} (assuming Proposition~\ref{prop:input}).]

Let $\mc{O}_0$ denote the set of orbit sets that do not include the binding $B$. It is enough to show that
\begin{equation}
\label{eqn:rets}
\inf\left\{\frac{\A(\alpha)}{\ell(\alpha,B)} \;\bigg|\; \alpha\in\mc{O}_0\right\} \le \sqrt{\theta_0V}.
\end{equation}
The reason is that if $\alpha=\{(\alpha_i,m_i)\}\in\mc{O}_0$ is an orbit set not including the binding, then by equations \eqref{eqn:aalpha} and \eqref{eqn:lalpha}, at least one of the orbits $\alpha_i$ must have
\[
\frac{\mc{A}(\alpha_i)}{\ell(\alpha_i,B)}\le \frac{\mc{A}(\alpha)}{\ell(\alpha,B)}.
\]

To prove \eqref{eqn:rets}, let $\epsilon>0$ so that
\begin{equation}
\label{eqn:epsilonsmall}
V+\epsilon < \theta_0.
\end{equation}
Choose $k>0$ sufficiently large with respect to $\epsilon$ so that Proposition~\ref{prop:input} applies to give an orbit set $\alpha\in\mc{O}_0$ and a nonnegative integer $m$ with
\begin{align}
\label{eqn:abound2}
(\A(\alpha)+m)^2 &\le 2k(V + \epsilon),\\
\label{eqn:lbound2}
\ell(\alpha,B)+m\theta_0^{-1} &\ge N_k(1,\theta_0^{-1}).
\end{align}
By Lemma~\ref{lem:comb}, the inequality \eqref{eqn:lbound2} implies that
\begin{equation}
\label{eqn:lbound3}
(\ell(\alpha,B) + m\theta_0^{-1})^2\ge 2k\theta_0^{-1} - ck^{1/2}
\end{equation}
where $c$ is a constant depending only on $\theta_0$ and not on $k$. Now \eqref{eqn:abound2} gives us an upper bound on $\A(\alpha)$, namely
\begin{equation}
\label{eqn:aub}
\A(\alpha) \le \sqrt{2k(V + \epsilon)} - m.
\end{equation}
And \eqref{eqn:lbound3} gives us a lower bound on $\ell(\alpha,B)$, namely
\begin{equation}
\label{eqn:lub}
\ell(\alpha,B) \ge \sqrt{2k\theta_0^{-1} - ck^{1/2}} - m\theta_0^{-1}.
\end{equation}
Note that the right hand side of \eqref{eqn:lub} is positive when $k$ is sufficiently large, because \eqref{eqn:abound2} and \eqref{eqn:epsilonsmall} imply that $m < h\sqrt{2k\theta_0}$ where $h<1$ is a constant independent of $k$. Assuming that $k$ is sufficiently large in this sense, we can then divide \eqref{eqn:aub} by \eqref{eqn:lub} to get
\begin{equation}
\label{eqn:mm}
\frac{\A(\alpha)}{\ell(\alpha,B)} \le \theta_0\frac{\sqrt{2k(V + \epsilon)} - m}{\sqrt{2k\theta_0 - c\theta_0^2k^{1/2}} - m}.
\end{equation}
If we regard the right hand side of \eqref{eqn:mm} as a function of the nonnegative integer $m$, then it is maximized when $m=0$, as long as
\begin{equation}
\label{eqn:ksl}
2k(V+\epsilon) \le 2k\theta_0 - c\theta_0^2k^{1/2}.
\end{equation}
The inequality \eqref{eqn:ksl} holds if $k$ is sufficiently large, by \eqref{eqn:epsilonsmall}, and we then obtain
\[
\frac{\A(\alpha)}{\ell(\alpha,B)} \le \sqrt{\frac{2k(V+\epsilon)}{2k\theta_0^{-1}-ck^{1/2}}}.
\]
Since $k$ can be arbitrarily large, taking $k\to\infty$ gives
\[
\inf\left\{\frac{\A(\alpha)}{\ell(\alpha,B)} \;\bigg|\; \alpha\in\mc{O}_0\right\} \le \sqrt{\theta_0(V+\epsilon)}.
\]
Since $\epsilon>0$ can be arbitrarily small, this proves \eqref{eqn:rets}.
\end{proof}

\begin{remark}
It may be possible to improve the upper bound in Proposition~\ref{prop:reduced} to $V$, and to drop the assumption that $\theta_0$ is irrational, similarly to the arguments in \S\ref{sec:twist}. However we do not need this.
\end{remark}

\section{Review of embedded contact homology}
\label{sec:review}

We now review some notions from embedded contact homology (ECH) that will be needed in the proof of Proposition~\ref{prop:input}. For more details about ECH, see \cite{bn} and the references therein. Readers familiar with ECH may wish to skip ahead to \S\ref{sec:filtration}.

\subsection{Definition of ECH}

Let $Y$ be a closed oriented three-manifold.
Let $\lambda$ be a contact form on $Y$. A Reeb orbit $\gamma$ is {\em nondegenerate\/} if the linearized return map $P_\gamma$ in \eqref{eqn:lrm} does not have $1$ as an eigenvalue. In this case $\gamma$ is either elliptic or {\em hyperbolic\/}, meaning that $P_\gamma$ has real eigenvalues. The contact form $\lambda$ is {\em nondegenerate\/} if all (not necessarily simple) Reeb orbits are nondegenerate. Generic contact forms have this property. Assume below that $\lambda$ is nondegenerate.

The {\em embedded contact homology\/} $ECH(Y,\lambda)$ is the homology of a chain complex $ECC(Y,\lambda,J)$ over $\Z/2$ defined as follows\footnote{It is also possible to define ECH over $\Z$, see \cite[\S9]{obg2}, but that is not necessary for this paper (or for any of the other applications thus far).}.

The chain complex $ECC(Y,\lambda,J)$ is freely generated over $\Z/2$ by orbit sets $\alpha=\{(\alpha_i,m_i)\}$ such that $m_i=1$ whenever $\alpha_i$ is hyperbolic. For some explanation of why we add this last condition, see \cite[\S\S2.7,5.4]{bn}.

If we assume that $H_1(Y)=0$, then the chain complex has an absolute $\Z$-grading, and we will denote it by $ECC_*$ to indicate this grading. The grading of a generator $\alpha=\{(\alpha_i,m_i)\}$ is given by
\begin{equation}
\label{eqn:grading}
I(\alpha) = -\sum_im_i\op{sl}(\alpha_i) + \sum_{i\neq j}m_im_j\ell(\alpha_i,\alpha_j) + \sum_i\sum_{k=1}^{m_i}\left(\floor{k\theta_i}+\ceil{k\theta_i}\right).
\end{equation}
Here $\op{sl}(\alpha_i)$ denotes the self-linking number\footnote{The self-linking number of a simple Reeb orbit $\gamma$, or more generally of any transverse knot (oriented knot that is positively transverse to $\xi$) can be defined as follows. Let $\tau$ denote the homotopy class of symplectic trivialization of $\xi|_\gamma$ for which a pushoff of $\gamma$ has linking number zero with $\gamma$. Let $\Sigma$ be a Seifert surface for $\gamma$. Then $\op{sl}(\gamma) = -c_1(\xi|_\Sigma,\tau)$, where the right hand side denotes the relative first Chern class, see \cite[\S3.2]{bn}. The grading formula \eqref{eqn:grading} is obtained from the more usual grading formula \cite[Eq.\ (3.11)]{bn} using the above trivialization $\tau$.} of $\alpha_i$, while $\ell$ denotes the linking number as usual, and $\theta_i$ denotes the rotation number $\op{rot}(\alpha_i)$.

We say that an almost complex structure $J$ on $\R\times Y$ is {\em $\lambda$-compatible\/} if it satisfies the following conditions:
\begin{itemize}
\item $J$ is $\R$-invariant.
\item
$J$ sends $\xi$ to itself, rotating positively with respect to $d\lambda$.
\item
 $J(\partial_s)=R$, where $s$ denotes the $\R$ coordinate and $R$ denotes the Reeb vector field. 
\end{itemize}
Fix a $\lambda$-compatible $J$. Let $\alpha=\{(\alpha_i,m_i)\}$ and $\beta=\{(\beta_j,n_j)\}$ be orbit sets. A {\em $J$-holomorphic curve\/} from $\alpha$ to $\beta$ is a $J$-holomorphic curve in $\R\times Y$ with positive ends at covers of $\alpha_i$ with total covering multiplicity $m_i$, negative ends at covers of $\beta_j$ with total covering multiplicity $n_j$, and no other ends. See \cite[\S3.1]{bn} for more details.

A {\em $J$-holomorphic current\/} from $\alpha$ to $\beta$ is a finite formal sum $\current = \sum_kd_kC_k$ where the $C_k$ are distinct irreducible somewhere injective $J$-holomorphic curves from orbit sets $\alpha(k)$ to $\beta(k)$, the $d_k$ are positive integers, and $\prod_k\alpha(k)^{d(k)}=\alpha$ and $\prod_k\beta(k)^{d(k)}=\beta$. Here the product of orbit sets is defined by adding the multiplicities of all Reeb orbits involved.  The curves $C_k$ are the {\em components\/} of the holomorphic current $\mc{C}$, and one can think of the integers $d_k$ as ``multiplicities''.  Let $\mc{M}^J(\alpha,\beta)$ denote the set of $J$-holomorphic currents from $\alpha$ to $\beta$.

To define the differential on the chain complex, we choose a suitably generic $\lambda$-compatible almost complex structure $J$. To simplify notation, we restrict to the case $H_1(Y)=0$, which is all that we will need (see \cite[\S3.5]{bn} for the general case). If $\alpha$ is a chain complex generator, we then define
\[
\partial\alpha = \sum_{I(\beta) = I(\alpha) - 1} \#\left(\mc{M}^J(\alpha,\beta)/\R\right)\beta
\]
where in the sum $\beta$ is another chain complex generator, the $\R$ action is by translation in $\R\times Y$, and $\#$ denotes the mod 2 count. See \cite[\S5.3]{bn} for the proof that the set being counted is finite. It is shown in \cite{obg1} that $\partial^2=0$. Thus we have a well-defined chain complex $ECC_*(Y,\lambda,J)$, and its homology is the embedded contact homology $ECH_*(Y,\lambda)$. It follows from the theorem of Taubes \cite{e1} relating ECH to Seiberg-Witten Floer cohomology \cite{km} that $ECH_*(Y,\lambda)$ does not depend on the choice of $J$; in fact it only depends on $Y$ and $\xi$.

\subsection{The ellipsoid example}
\label{sec:ellipsoid}

In the present paper we will need a detailed understanding of the ECH chain complex for the boundary of an ellipsoid.

If $Y$ is a star-shaped hypersurface in $\R^4$, then the Liouville form
\[
\lambda_0 = \frac{1}{2}\sum_{j=1}^2\left(x_jdy_j - y_jdx_j\right)
\]
restricts to a contact form on $Y$. In particular, let $Y$ be the boundary of the ellipsoid
\[
E(a,b) = \left\{z\in\C^2\;\bigg|\;\frac{\pi|z_1|^2}{a} + \frac{\pi|z_2|^2}{b} \le 1\right\}
\]
where $a,b>0$ are constants with $a/b$ irrational. Of course $Y$ is diffeomorphic to $S^3$. We will need the following explicit description of the chain complex $ECC_*(Y,{\lambda_0}|_Y,J)$, for any $\lambda_0$-compatible almost complex structure $J$.

\begin{lemma}
\label{lem:ellipsoid}
Let $Y=\partial E(a,b)$ where $a/b$ is irrational. Then:
\begin{description}
\item{(a)}
The Liouville form $\lambda_0$ restricts to a nondegenerate contact form on $Y$.
\item{(b)}
The grading defines a bijection from the set of generators of $ECC_*(Y,\lambda_0,J)$ to the set of nonnegative even integers.
\item{(c)}
The grading $2k$ generator has symplectic action $N_k(a,b)$.
\end{description}
\end{lemma}

\begin{proof}
The Reeb vector field on $Y$ is given by
\[
R = {2\pi}\left(\frac{1}{a}\frac{\partial}{\partial \theta_1} + \frac{1}{b}\frac{\partial}{\partial \theta_2}\right)
\]
where $\theta_j$ denotes the angular polar coordinate on the $x_j,y_j$ plane. It follows that there are just two simple Reeb orbits: the circle $\gamma_1=(z_2=0)$, and the circle $\gamma_2=(z_1=0)$.  These orbits have symplectic action $a$ and $b$ respectively. The orbit $\gamma_1$ is elliptic with rotation number $a/b$, and the orbit $\gamma_2$ is elliptic with rotation number $b/a$. In particular, all covers of these orbits are nondegenerate since $a/b$ is irrational. This proves (a).

The chain complex generators now have the form $\gamma_1^{m}\gamma_2^n$ where $m,n$ are nonegative integers\footnote{The notation $\gamma_1^m\gamma_2^n$ is shorthand for the orbit set containing $(\gamma_1,m)$ when $m>0$ and $(\gamma_2,n)$ when $n>0$.}. The action of a generator is given by
\begin{equation}
\label{eqn:aog}
\A(\gamma_1^m\gamma_2^n)=am+bn.
\end{equation}
To compute the grading, note that the orbits $\gamma_1$ and $\gamma_2$ both have self-linking number $-1$, and their linking number is $1$. It then follows from the grading formula \eqref{eqn:grading} that
\[
I(\gamma_1^m\gamma_2^n) = 2\left((m+1)(n+1)-1+\sum_{i=1}^m\floor{ia/b} + \sum_{j=1}^n\floor{jb/a}\right).
\]
It follows from this and equation \eqref{eqn:latticepoints} that
\[
I(\gamma_1^m\gamma_2^n) = 2k \Longleftrightarrow N_k(a,b)=am+bn.
\]
Together with \eqref{eqn:aog}, this implies (b) and (c).
\end{proof}

\subsection{Filtered ECH and the ECH spectrum}
\label{sec:filtered}

Let $Y$ be a closed oriented three-manifold, let $\lambda$ be a nondegenerate contact form on $Y$, and let $J$ be a generic $\lambda$-compatible almost complex structure on $\R\times Y$ as needed to define the ECH chain complex $ECC(Y,\lambda,J)$. It follows from the definition of $\lambda$-compatible almost complex structure that if $\alpha$ and $\beta$ are orbit sets and if the moduli space $\mc{M}^J(\alpha,\beta)$ is nonempty, then $\mc{A}(\alpha)\ge \mc{A}(\beta)$, with equality only if $\alpha=\beta$. In particular, the differential decreases symplectic action. Thus for each $L\in\R$ we have a subcomplex $ECC^L(Y,\lambda,J)$, defined to be the span of those chain complex generators $\alpha$ for which $\mc{A}(\alpha)<L$. The homology of this subcomplex is the {\em filtered ECH\/}, denoted by $ECH^L(Y,\lambda)$. It is shown in \cite[Thm.\ 1.3]{cc2} that filtered ECH does not depend on $J$. However it does depend on the contact form $\lambda$ and not just on the contact structure $\xi$.

If $r>0$ is a constant then there is a canonical ``scaling'' isomorphism
\begin{equation}
\label{eqn:scalingiso}
ECH^L(Y,\lambda) = ECH^{rL}(Y,r\lambda).
\end{equation}
To see this, note that $\lambda$ and $r\lambda$ have the same Reeb orbits up to reparametrization. If $J$ is a generic $\lambda$-compatible almost complex structure, then there is a unique $r\lambda$-compatible almost complex structure $J^r$ which agrees with $J$ on the contact planes; and $J$-holomorphic curves correspond to $J^r$-holomorphic curves under rescaling the $\R$ coordinate on $\R\times Y$ by $r$. Thus the above bijection on Reeb orbits gives an isomorphism of chain complexes
\begin{equation}
\label{eqn:scalingisochain}
ECC^L(Y,\lambda,J) = ECC^{rL}(Y,r\lambda,J^r).
\end{equation}
There are also maps
\begin{equation}
\label{eqn:inclusion1}
ECH^L(Y,\lambda) \longrightarrow ECH(Y,\lambda)
\end{equation}
and
\begin{equation}
\label{eqn:inclusion2}
ECH^L(Y,\lambda) \longrightarrow ECH^{L'}(Y,\lambda)
\end{equation}
for $L\le L'$
induced by inclusion of chain complexes. It is shown in \cite[Thm.\ 1.3]{cc2} that the maps \eqref{eqn:scalingiso}, \eqref{eqn:inclusion1} and \eqref{eqn:inclusion2} do not depend on $J$.

We now introduce the ECH spectrum in the special case that we will need, which is where $\lambda$ is a nondegenerate contact form on $S^3$ whose kernel is the standard tight contact structure (given by the kernel of the Liouville form $\lambda_0$ on a compact star-shaped hypersurface in $\R^4$). We know from Lemma~\ref{lem:ellipsoid} that
\[
ECH_*(S^3,\lambda) = \left\{\begin{array}{cl} \Z/2, & *=0,2,4,\ldots,\\
0, & \mbox{otherwise.}
\end{array}\right.
\]
We define $c_k(S^3,\lambda)$ to be the infimum over $L$ such that the generator of $ECH_{2k}(S^3,\lambda)$ is in the image of the inclusion-induced map \eqref{eqn:inclusion1}. Equivalently, if we choose an almost complex structure $J$ as needed to define the ECH differential, then $c_k(S^3,\lambda)$ is the minimum over $L$ such that the generator of $ECH_{2k}(S^3,\lambda)$ can be represented by a cycle in $ECC_{2k}(S^3,\lambda,J)$ such that each generator in the cycle has symplectic action at most $L$. The list of numbers $\{c_k(S^3,\lambda)\}_{k=0,1,\ldots}$ is called the {\em ECH spectrum\/} of $(S^3,\lambda)$.

The ECH spectrum behaves as follows if one multiplies the contact form by a positive function. First, if $r>0$ is a constant, then
\begin{equation}
\label{eqn:scaling}
c_k(S^3,r\lambda) = rc_k(S^3,\lambda).
\end{equation}
Second, if $f:Y\to\R$ is a function with $f\ge 1$ everywhere (and if $f\lambda$ is nondegenerate), then
\begin{equation}
\label{eqn:monotone}
c_k(S^3,f\lambda) \ge c_k(S^3,\lambda).
\end{equation}
The scaling property \eqref{eqn:scaling} can be proved using the isomorphism \eqref{eqn:scalingiso}, and the monotonicity property \eqref{eqn:monotone} can be proved using the cobordism maps reviewed in \S\ref{sec:cobordism}. For details see \cite[\S4]{qech}.

Given the above properties, one can extend the definition of $c_k$ to degenerate contact forms as follows. If $\lambda$ is degenerate, let $\{f_n:Y\to\R\}_{n=1,2,\ldots}$ be a sequence of positive functions with $f_n\to 1$ in the $C^0$ topology such that $f_n\lambda$ is nondegenerate for each $n$. We then define
\begin{equation}
\label{eqn:limit}
c_k(S^3,\lambda) = \lim_{n\to \infty}c_k(S^3,f_n\lambda).
\end{equation}
The properties \eqref{eqn:scaling} and \eqref{eqn:monotone} imply that \eqref{eqn:limit} is well defined and again satisfies properties \eqref{eqn:scaling} and \eqref{eqn:monotone}.

The key nontrivial fact about the ECH spectrum that we will need is the following:

\begin{theorem}
\label{thm:vc}
Let $\lambda$ be a contact form on $S^3$ whose kernel is the standard tight contact structure. Then
\[
\lim_{k\to\infty}\frac{c_k(S^3,\lambda)^2}{2k} = \op{vol}(S^3,\lambda).
\]
\end{theorem}

\begin{example}
\label{ex:evc}
If $(S^3,\lambda)$ is diffeomorphic to $(\partial E(a,b),\lambda_0)$, then it follows from Lemmas~\ref{lem:ellipsoid} and \ref{lem:comb} that
\[
\lim_{k\to\infty}\frac{c_k(S^3,\lambda)^2}{2k} = ab.
\]
The right hand side agrees with $\op{vol}(\partial E(a,b),\lambda_0)$, confirming the theorem in this case.
\end{example}

\begin{proof}[Proof of Theorem~\ref{thm:vc}.]
This is a special case of \cite[Thm.\ 1.3]{vc}, which is a more general result valid for any contact three-manifold, proved using Seiberg-Witten theory. One can also prove this particular case without using Seiberg-Witten theory (except as currently needed to define ECH cobordism maps) by using \cite[Prop.\ 8.6(b)]{qech} to reduce to Example~\ref{ex:evc}.
\end{proof}

\section{Knot filtration on ECH}
\label{sec:filtration}

We now define a new ``knot filtration'' on ECH and use it to prove Proposition~\ref{prop:input}.

Let $Y$ be a closed three-manifold with $H_1(Y)=0$.
Let $\lambda$ be a nondegenerate contact form on $Y$, and let $J$ be an generic $\lambda$-compatible almost complex structure on $\R\times Y$. Let $B$ be a simple elliptic Reeb orbit with rotation number $\theta_0\in\R\setminus\Q$. (This $\theta_0$ corresponds to the inverse of the $\theta_0$ considered previously.) We can use $B$ to define a filtration $\mc{F}_B$ on the ECH chain complex $ECC_*(Y,\lambda,J)$ as follows.

Consider an orbit set $B^m\alpha$ where $m$ is a nonnegative integer and $\alpha$ is an orbit set\footnote{Here we are not assuming that $B^m\alpha$ is a generator of the ECH chain complex, i.e.\ we do not require that hyperbolic orbits have multiplicity one.} not including $B$. We define
\[
\mc{F}_B(B^m\alpha) = m\theta_0 + \ell(\alpha,B).
\]
Of course $\mc{F}_B$ is not integer valued. However if $\theta_0>0$ and if every Reeb orbit other than $B$ has nonnegative linking number with $B$ (for example in the open book setting of Proposition~\ref{prop:reduced} where this linking number is always positive), then $\mc{F}_B$ takes values in a discrete set of nonnegative real numbers.

\begin{lemma}
\label{lem:filtration}
If $B^{m_+}\alpha_+$ and $B^{m_-}\alpha_-$ are orbit sets, and if there exists a $J$-holomorphic current $\current\in\mc{M}^J(B^{m_+}\alpha_+,B^{m_-}\alpha_-)$, then
\begin{equation}
\label{eqn:F}
\mc{F}_B(B^{m_+}\alpha_+) \ge \mc{F}_B(B^{m_-}\alpha_-).
\end{equation}
In particular, the ECH differential $\partial$ does not increase $\mc{F}_B$.
\end{lemma}

\begin{remark}
The definition of ``$\lambda$-compatible'' almost complex structure implies that the ``trivial cylinder'' $\R\times B$ is automatically a $J$-holomorphic cylinder in $\R\times Y$.
The proof of Lemma~\ref{lem:filtration} will show that equality holds in \eqref{eqn:F} if and only if $m_+=m_-$, the current $\mc{C}$ includes $\R\times B$ with multiplicity $m_+$ when $m_+>0$, and no other component of $\mc{C}$ intersects $\R\times B$ or has an end at a cover of $B$.
\end{remark}

\begin{proof}[Proof of Lemma~\ref{lem:filtration}.]
By linearity of all the terms in \eqref{eqn:F}, we may assume without loss of generality that $\current$ consists of a single irreducible somewhere injective component $C$ which does not agree with $\R\times B$.

By standard results on the asymptotics of holomorphic curves, see e.g.\ \cite[Cor.\ 2.5, 2.6]{siefring} for quite strong versions of these, if $s_0>0$ is sufficiently large, then $C$ is transverse to $\{\pm s_0\}\times Y$, and $C\cap((-\infty,-s_0]\times Y)$ and $C\cap([s_0,\infty)\times Y)$ do not intersect $\R\times B$. Choose $s_0$ sufficiently large in this sense and let $\eta_\pm$ denote the intersection of $C$ with $\{\pm s_0\}\times Y$. Observe that
\[
\ell(\eta_+,B) - \ell(\eta_-,B) = \#(C\cap(\R\times B)).
\] 
Here `$\#$' denotes algebraic intersection number. By intersection positivity of the $J$-holomorphic curve $C$ with the $J$-holomorphic cylinder $\R\times B$, the right hand side is nonnegative, so
\begin{equation}
\label{eqn:elletaB}
\ell(\eta_+,B) \ge \ell(\eta_-,B)
\end{equation}
(with equality if and only if $C$ does not intersect $\R\times B$).

The link $\eta_\pm$ consists of a link approximating $\alpha_\pm$, together with a link $\zeta_\pm$ in a neighborhood of $B$. We then have
\begin{equation}
\label{eqn:elleaz}
\ell(\eta_\pm,B) = \ell(\alpha_\pm,B) + \ell(\zeta_\pm,B).
\end{equation}

The link $\zeta_+$ has one component for each positive end of $C$ at a cover of $B$. Let $q_1,\ldots,q_k$ denote the covering multiplicities of these ends, so that $\sum_{i=1}^kq_i=m_+$, and let $\zeta_+^1,\ldots,\zeta_+^k$ denote the corresponding components of $\zeta_+$. By results on the asymptotics of holomorphic curves going back to \cite[\S3]{hwz} and reviewed in \cite[Lem.\ 5.3(b)]{bn}, we have
\[
\ell(\zeta_+^i,B) \le \floor{q_i\theta_0}.
\]
Therefore
\[
\begin{split}
\ell(\zeta_+,B) 
%&= \sum_{i=1}^k\ell(\zeta_+^i,B)\\
&\le \sum_{i=1}^k\floor{q_i\theta_0}\\
&\le \floor{m_+\theta_0}.
\end{split}
\]
It follows that
\begin{equation}
\label{eqn:ellzeta+B}
\ell(\zeta_+,B) \le m_+\theta_0.
\end{equation}
(Since $\theta_0$ is irrational, equality holds if and only if $m_+=0$.) Similarly,
\begin{equation}
\label{eqn:ellzeta-B}
\ell(\zeta_-,B) \ge m_-\theta_0
\end{equation}
(with equality if and only if $m_-=0$). Combining \eqref{eqn:elletaB}, \eqref{eqn:elleaz}, \eqref{eqn:ellzeta+B} and \eqref{eqn:ellzeta-B} proves that \eqref{eqn:F} holds.
\end{proof}

If $R$ is a real number, let $ECH_*^{\mc{F}_B\le R}(Y,\lambda,J)$ denote the homology of the subcomplex generated by admissible orbit sets $B^m\alpha$ with $\mc{F}_B(B^m\alpha)\le R$. The following theorem asserts that this is a topological invariant. (In the special case when $\mc{F}_B$ takes values in a discrete set, the associated spectral sequence is also a topological invariant, by a similar argument.)

\begin{theorem}
\label{thm:filtration}
Let $Y$ be a closed three-manifold with $H_1(Y)=0$, let $\xi$ be a contact structure on $Y$, let $B\subset Y$ be a transverse knot, let $\theta_0\in\R\setminus \Q$, and let $R\in\R$. Let $\lambda$ be a contact form\footnote{Such a contact form always exists.} with $\Ker(\lambda)=\xi$ such that $B$ is an elliptic Reeb orbit with rotation number $\theta_0$. Let $J$ be any generic $\lambda$-compatible almost complex structure. Then $ECH_*^{\mc{F}_B\le R}(Y,\lambda,J)$ depends only on $Y$, $\xi$, $B$, $\theta_0$, and $R$.
\end{theorem}

By Theorem~\ref{thm:filtration}, we can denote $ECH_*^{\mc{F}_B\le R}(Y,\lambda,J)$ by $ECH_*^{\mc{F}\le R}(Y,\xi,B,\theta_0)$. The proof of Theorem~\ref{thm:filtration} requires some additional technical preliminaries and is deferred to \S\ref{sec:functoriality}. Meanwhile, here is the key example that we will need.

\begin{example}
\label{ex:main}
Let $Y$ be the boundary of the ellipsoid $E(a,b)$
where $a,b>0$ are constants with $a/b$ irrational, with the restriction of the Liouville form $\lambda_0$, and let $B=\gamma_2$, see \S\ref{sec:ellipsoid}.
 Recall that $\gamma_2$ has rotation number $\theta_0=b/a$ and linking number $1$ with $\gamma_1$. The filtration of an ECH generator $\gamma_1^{d}\gamma_2^m$, where $d,m\in\N$, is then given by
\[
\begin{split}
\mc{F}_{\gamma_2}(\gamma_1^d\gamma_2^m) &= d + mb/a\\
&= a^{-1}\mc{A}(\gamma_1^d\gamma_2^m).
\end{split}
\]
It follows using Lemma~\ref{lem:ellipsoid} that if $x$ is an ECH generator with grading $I(x)=2k$ then
\[
\mc{F}_{\gamma_2}(x) = N_k\left(1,b/a\right).
\]
We conclude that if $k$ is a nonnegative integer then
\[
ECH_{2k}^{\mc{F}_{\gamma_2}\le R}(\partial E(a,b),\lambda_0,b/a) = \left\{\begin{array}{cl} \Z/2, & R\ge N_k(1,b/a),\\
0, & \mbox{otherwise}
\end{array}\right.
\]
and $ECH_*^{\mc{F}_{\gamma_2}\le R}=0$ for all other gradings $*$.
\end{example}

Example~\ref{ex:main} and the topological invariance in Theorem~\ref{thm:filtration} imply the following:

\begin{proposition}
\label{prop:filtcomp}
Let $\xi_0$ be the standard tight contact structure on $S^3$, let $B_0$ be the standard transverse unknot (given by a Hopf circle\footnote{Up to isotopy of transverse knots, this is the unique transverse unknot with self-linking number $-1$, see \cite{yasha,etnyre}.}), let $R\in\R$, and let $\theta_0>0$ be irrational. Then for $k\in\N$ we have
\[
ECH_{2k}^{\mc{F}\le R}(S^3,\xi_0,B_0,\theta_0) = \left\{\begin{array}{cl}
\Z/2, & R\ge N_k(1,\theta_0),\\
0, & \mbox{otherwise}
\end{array}\right.
\]
and
\[
ECH_{*}^{\mc{F}\le R}(S^3,\xi_0,B_0,\theta_0) = 0
\]
for all other gradings $*$.
\end{proposition}

We can now give:

\begin{proof}[Proof of Proposition~\ref{prop:input} (assuming Theorem~\ref{thm:filtration}).] The proof has two steps.

{\em Step 1.\/} We first prove the proposition in the case when $\lambda$ is nondegenerate.

By Theorem~\ref{thm:vc}, if $k$ is sufficiently large then
\begin{equation}
\label{eqn:vcinput}
\frac{c_k(S^3,\lambda)^2}{2k} \le \op{vol}(S^3,\lambda) + \epsilon.
\end{equation}
Assume that $k$ is sufficiently large in this sense.

Let $J$ be a generic $\lambda$-compatible almost complex structure on $\R\times S^3$. By the definition of $c_k(S^3,\lambda)$, there exists a cycle
\[
x=\sum_ix_i\in ECC_{2k}(S^3,\lambda,J)
\]
which represents the generator of $ECH_{2k}(S^3,\lambda)$, such that each orbit set $x_i$ in the cycle has symplectic action $\A(x_i) \le c_k(S^3,\lambda)$. In particular, for each $i$ we have
\begin{equation}
\label{eqn:input1}
\frac{\A(x_i)^2}{2k} \le \op{vol}(S^3,\lambda) + \epsilon.
\end{equation}

Now for at least one $i$, we must also have
\begin{equation}
\label{eqn:input2}
\mc{F}_B(x_i) \ge N_k(1,\op{rot}(B)).
\end{equation}
Otherwise, by Proposition~\ref{prop:filtcomp}, $x$ would be nullholomogous, contradicting the fact that $x$ represents the generator of $ECH_{2k}(S^3,\lambda)$.

Choose $i$ such that $x_i$ satisfies both \eqref{eqn:input1} and \eqref{eqn:input2}. Write $x_i=B^m\alpha$ where $m$ is a nonnegative ineger and $\alpha$ is an orbit set that does not include $B$. Then the inequalities \eqref{eqn:input1} and \eqref{eqn:input2} translate into the desired conclusions \eqref{eqn:abound} and \eqref{eqn:lbound} respectively.

{\em Step 2.\/} We now drop the assumption that $\lambda$ is nondegenerate.

If $\lambda$ is degenerate, let $\{f_n\}_{n=1,2,\ldots}$ be a sequence of functions $S^3\to\R$ with $0<f_n\le 1$ such that $f_n\to 1$ in $C^\infty$ and the contact form $\lambda_n=f_n\lambda$ is nondegenerate. We can arrange that each $\lambda_n$ satisfies the hypotheses of the proposition for each $n$, with the same binding orbit $B$, such that $\A(B)$ and $\op{rot}(B)$ do not depend on $n$.

We know from Theorem~\ref{thm:vc} (which does not assume nondegeneracy of $\lambda$) that if $k$ is sufficiently large then \eqref{eqn:vcinput} holds.  Assume that $k$ is sufficiently large in this sense. It then follows from the monotonicity property \eqref{eqn:monotone} that
\[
\frac{c_k(S^3,f_n\lambda)^2}{2k} \le \op{vol}(S^3,\lambda) + \epsilon
\]
for each $n$.
By the argument in Step 1, for each $n$ we can choose a nonnegative integer $m(n)$, and an orbit set $\alpha(n)$ not including $B$, such that
\begin{align}
\label{eqn:deg1}
\frac{(\A_n(\alpha(n)) + m(n)\A(B))^2}{2k} & \le \op{vol}(S^3,\lambda) + \epsilon,\\
\label{eqn:deg2}
\ell(\alpha(n),B) + m(n)\op{rot}(B) & \ge N_k(1,\op{rot}(B)).
\end{align}
Here $\A_n$ denotes the action determined by $\lambda_n$. 

The inequality \eqref{eqn:deg1} implies that there is an $n$-independent upper bound on the nonnegative integer $m(n)$. Thus we may pass to a subsequence so that $m(n)$ is constant; denote this constant by $m$. The inequality \eqref{eqn:deg1} also gives an $n$-independent upper bound on $\A_n(\alpha(n))$. On the other hand there is an $n$-independent lower bound on the action of any Reeb orbit for $\lambda_n$ (since the Reeb vector fields of $\lambda_n$ and $\lambda$ are nonsingular). It follows using \eqref{eqn:deg1} again that the total multiplicity of all Reeb orbits in $\alpha(n)$ has an $n$-independent upper bound. We can then pass to a further subsequence so that $\alpha(n)=\{\alpha_i(n),m_i\}$ where $m_i$ is an $n$-independent positive integer, and $\alpha_i(n)$ is a simple Reeb orbit for $\lambda_n$ with $\lim_{n\to\infty}\alpha_i(n)=\alpha_i$, where $\alpha_i$ is a Reeb orbit for $\lambda$. Since the contact form $\lambda$ is degenerate, the Reeb orbits $\alpha_i$ need not be simple or distinct. Nonetheless, since the binding $B$ is a nondegenerate Reeb orbit of $\lambda$, we do know that none of the orbits $\alpha_i$ is a cover of $B$. Thus $\alpha=\prod_i\alpha_i^{m_i}$ is a well-defined orbit set for $\lambda$ that does not include $B$, and the sequence of orbit sets $\alpha(n)$ converges ``as a current'' to $\alpha$. In particular $\A_n(\alpha(n))$ converges to $\A(\alpha)$, and $\ell(\alpha(n),B)$ is equal to $\ell(\alpha,B)$ for $n$ sufficiently large. Thus the inequalities \eqref{eqn:deg1} and \eqref{eqn:deg2} imply that the pair $(\alpha,m)$ satisfies the conclusions of the proposition.
\end{proof}

\section{Cobordism maps on ECH}
\label{sec:cobordism}

To prepare for the proof of Theorem~\ref{thm:filtration}, we now review some properties of maps on ECH induced by exact symplectic cobordisms. (For maps on ECH induced by more general strong symplectic cobordisms, see \cite[\S2.4]{qech} and more generally \cite{field}.)

\subsection{Exact symplectic cobordisms}

Let $(Y_+,\lambda_+)$ and $(Y_-,\lambda_-)$ be closed three-manifolds with contact forms. An {\em exact symplectic cobordism\/} from $(Y_+,\lambda_+)$ to $(Y_-,\lambda_-)$ is a pair $(X,\lambda)$, where $X$ is a compact four-manifold with boundary\footnote{Strictly speaking, we should say that an exact symplectic cobordism includes a choice of diffeomorphism $\partial X \simeq Y_+ \sqcup -Y_-$.} $\partial X = Y_+ - Y_-$, and $\lambda$ is a one-form on $X$ such that $d\lambda$ is symplectic and $\lambda|_{Y_\pm}=\lambda_\pm$.

Given an exact symplectic cobordism $(X,\lambda)$, a neighborhood of $Y_+$ in $X$ can be canonically identified with $(-\epsilon,0]\times Y_+$ for some $\epsilon>0$ such that $\lambda$ is identified with $e^s\lambda_+$, where $s$ denotes the $(-\epsilon,0]$ coordinate. This identification is defined so that $\partial_s$ corresponds to the unique vector field $V$ such that $\imath_Vd\lambda = \lambda$. Likewise, a neighborhood of $Y_-$ in $X$ can be canonically identified with $[0,\epsilon)\times Y_-$ so that $\lambda$ is identified with $e^s\lambda_-$. Using these identifications, we then glue to form the {\em completion\/}
\[
\overline{X} = ((-\infty,0]\times Y_-) \sqcup_{Y_-} X \sqcup_{Y_+} ([0,\infty)\times Y_+).
\]

If $(X_+,\lambda_+)$ is an exact symplectic cobordism from $(Y_1,\lambda_1)$ to $(Y_0,\lambda_0)$, and if $(X_-,\lambda_-)$ is an exact symplectic cobordism from $(Y_0,\lambda_0)$ to $(Y_{-1},\lambda_{-1})$, then we can similarly glue along $Y_0$ to define the {\em composition\/}
\[
X_-\circ X_+ = X_- \sqcup_{Y_0} X_+.
\]
The one-forms $\lambda_-$ and $\lambda_+$ glue to a one-form on $X_-\circ X_+$ which we denote by $\lambda_-\circ\lambda_+$ and which makes the latter into an exact symplectic cobordism from $(Y_1,\lambda_1)$ to $(Y_{-1},\lambda_{-1})$.

More generally, for $R\ge 0$, we can define the {\em stretched composition\/}
\[
X_- \circ_R X_+ = X_- \sqcup_{Y_0} ([-R,R]\times Y_0) \sqcup X_+.
\]
The one-forms $e^{\pm R}\lambda_\pm$ on $X_\pm$ and $e^s\lambda_0$ on $[-R,R]\times Y_0$ glue to a one-form $\lambda_-\circ_R\lambda_+$ on $X_-\circ_R X_+$ which makes the latter into an exact symplectic cobordism from $(Y_1,e^R\lambda_1)$ to $(Y_{-1},e^{-R}\lambda_{-1})$.

\subsection{Broken holomorphic currents}

Let $(X,\lambda)$ be an exact symplectic cobordism from $(Y_+,\lambda_+)$ to $(Y_-,\lambda_-)$. We say that an almost complex structure $J$ on the completion $\overline{X}$ is {\em cobordism-admissible\/} if $J$ is $d\lambda$-compatible on $X$ in the usual sense\footnote{That is, $g(\cdot,\cdot) = d\lambda(\cdot,J\cdot)$ is a Riemannian metric on $X$.}, and if there are $\lambda_\pm$-compatible almost complex structures $J_\pm$ on $\R\times Y_\pm$ such that $J$ agrees with $J_+$ on $[0,\infty)\times Y_+$ and with $J_-$ on $(-\infty,0]\times Y_-$.

\begin{remark}
\label{rem:composition}
One can compose cobordism-admissible almost complex structures as follows. Suppose that $(X_+,\lambda_+)$ is an exact symplectic cobordism from $(Y_1,\lambda_1)$ to $(Y_0,\lambda_0)$, and $(X_-,\lambda_-)$ is an exact symplectic cobordism from $(Y_0,\lambda_0)$ to $(Y_{-1},\lambda_{-1})$. Let $J_i$ be a $\lambda_i$-compatible almost complex structure on $\R\times Y_i$ for $i=-1,0,1$. Let $J_\pm$ be cobordism-admissible almost complex structures on the completions $\overline{X_\pm}$ that restrict to $J_{\pm1}$ and $J_0$ on the ends. 
 We can then glue $J_-$, $J_0$, and $J_+$ to define an almost complex structure $J_-\circ_R J_+$ on $\overline{X_- \circ_R X_+}$ for each $R\ge 0$.  When $R=0$, this is a cobordism-admissible almost complex structure on $\overline{X_-\circ X_+}$ which we denote simply by $J_-\circ J_+$. (When $R>0$, the almost complex structure $J_-\circ_R J_+$ is not quite cobordism-admissible, because the contact forms have been rescaled.)
\end{remark}

Now let $(X,\lambda)$ be an exact symplectic cobordism from $(Y_+,\lambda_+)$ to $(Y_-,\lambda_-)$, and fix a cobordism-admissible almost complex structure $J$ on the completion $\overline{X}$. If $\alpha_\pm$ are orbit sets for $\lambda_\pm$, we can define a set $\M^J(\alpha_+,\alpha_-)$ of $J$-holomorphic currents from $\alpha_+$ to $\alpha_-$ as before.

A {\em broken $J$-holomorphic current\/} from $\alpha_+$ to $\alpha_-$ is a tuple $\current=(\current_{N_-},\ldots,\current_{N_+})$ where $N_-\le 0\le N_+$ such that there are distinct orbit sets $\alpha_-=\alpha_-(N_-),\ldots,\alpha_-(0)$ for $\lambda_-$ and $\alpha_+(0),\ldots,\alpha_+(N_+)=\alpha_+$ for $\lambda_+$ such that:
\begin{itemize}
\item
$\current_0\in\M^J(\alpha_+(0),\alpha_-(0))$.
\item
If $k>0$ then $\current_k\in\M^{J_+}(\alpha_+(k),\alpha_+(k-1))/\R$.
\item
If $k<0$ then $\current_k\in\M^{J_-}(\alpha_-(k+1),\alpha_-(k))/\R$.
\end{itemize}
The currents $\current_k$ are called the {\em levels\/} of $\current$.
We denote the set of broken holomorphic currents from $\alpha_+$ to $\alpha_-$ by $\overline{\M^J(\alpha_+,\alpha_-)}$.

Suppose now that $\lambda_\pm$ is nondegenerate and $J_\pm$ is generic so that the chain complex $ECC(Y_\pm,\lambda_\pm,J_\pm)$ is defined.
We say that a linear map
\begin{equation}
\label{eqn:cjc}
\phi:ECC(Y_+,\lambda_+,J_+)\longrightarrow ECC(Y_-,\lambda_-,J_-)
\end{equation}
{\em counts $J$-holomorphic currents\/} if $\langle\phi\alpha_+,\alpha_-\rangle\neq 0$ implies that the set $\overline{\M^J(\alpha_+,\alpha_-)}$ is nonempty.

\subsection{Cobordism maps}

It was shown in \cite[Thm.\ 1.9]{cc2} that an exact symplectic cobordism $(X,\lambda)$ from $(Y_+,\lambda_+)$ to $(Y_-,\lambda_-)$, where $\lambda_\pm$ are nondegenerate, induces canonical maps
\begin{equation}
\label{eqn:cc2}
\Phi^L(X,\lambda): ECH^L(Y_+,\lambda_+) \longrightarrow ECH^L(Y_-,\lambda_-)
\end{equation}
satisfying various properties. One key property is the ``Holomorphic Curves Axiom'' asserting that if $J$ is a cobordism-admissible almost complex structure on $\overline{X}$ which restricts to generic $\lambda_\pm$-compatible almost complex structures $J_\pm$ on the ends, then $\Phi^L$ is induced by a (noncanonical) chain map as in \eqref{eqn:cjc} which counts $J$-holomorphic currents. This property is important because in all of the applications so far, it is the existence of a holomorphic curve which allows one to draw geometric conclusions.

We will need the following elaboration of this property. Below, if $\lambda$ is a contact form on $Y$ with Reeb vector field $R$, if $J$ is a $\lambda$-compatible almost complex structure on $\R\times Y$, and if $f:\R\to\R$ is a positive function, let $J^f$ denote the almost complex structure on $\R\times Y$ that agrees with $J$ on the contact structure $\xi$ and satisfies $J(\partial_s) = f(s)^{-1}R$. Also, following \cite[\S5.1]{cc2}, define a {\em strong homotopy\/} of exact symplectic cobordisms from $(Y_+,\lambda_+)$ to $(Y_-,\lambda_-)$ to be a pair $(X,\{\lambda_t\}_{t\in[0,1]})$ where $X$ is a compact four-manifold with boundary $\partial X = Y_+ - Y_-$, and $\{\lambda_t\}$ is a smooth family of $1$-forms on $X$ that is independent of $t$ near $\partial X$, such that for each $t$, the form $d\lambda_t$ is symplectic and $\lambda_t|_{Y_\pm} = \lambda_\pm$.

\begin{proposition}
\label{prop:elaboration}
Let $(Y_+,\lambda_+)$ and $(Y_-,\lambda_-)$ be closed three-manifolds with nondegenerate contact forms. Let $(X,\lambda)$ be an exact symplectic cobordism from $(Y_+,\lambda_+)$ to $(Y_-\lambda_-)$. Let
$J$ be a cobordism-admissible almost complex structure on $\overline{X}$ which restricts to generic $\lambda_\pm$-compatible almost complex structures $J_\pm$ on the ends. Then for each $L>0$ there exists a nonempty set $\Theta^L(X,\lambda,J)$ of chain maps
\begin{equation}
\label{eqn:cc2chain}
\phi: ECC^L(Y_+,\lambda_+,J_+) \longrightarrow ECC^L(Y_-,\lambda_-,J_-)
\end{equation}
satisfying the following properties:
\begin{description}
\item{(Holomorphic Curves)} Each $\phi\in \Theta^L(X,\lambda,J)$ counts $J$-holomorphic currents.
%\item{(Uniqueness)} If $\phi,\phi'\in\Theta^L(X,\lambda,J)$, then there exists a chain homotopy
%\[
%K: ECC^L(Y_+,\lambda_+,J_+) \longrightarrow ECC^L(Y_-,\lambda_-,J_-)
%\]
%such that\footnote{Here $\partial_\pm$ denotes the differential on the chain complex $ECC(Y_\pm,\lambda_\pm,J_\pm)$. We will use similar notation below without comment.}
%\[
%\partial_- K + K \partial_+ = \phi - \phi',
%\]
%and $K$ counts $J$-holomorphic currents.
\item{(Homotopy Invariance)}
Let $(X,\{\lambda_t\}_{t\in[0,1]})$ be a strong homotopy of exact symplectic cobordisms from $(Y_+,\lambda_+)$ to $(Y_-,\lambda_-)$. Let $\{J_t\}_{t\in[0,1]}$ be a family of almost complex structures on $\overline{X}$ such that for each $t$, we have that $J_t$ is cobordism-admissible for $\lambda_t$, and $J_t$ restricts to $J_\pm$. Given
\[
\phi_i\in\Theta^L(X,\lambda_i,J_i)
\]
for $i=0,1$, there is a map
\[
K: ECC^L(Y_+,\lambda_+,J_+) \longrightarrow ECC^L(Y_-,\lambda_-,J_-)
\]
which counts $J_t$-holomorphic currents\footnote{That is, if $\langle K\alpha(1),\alpha(0)\neq 0$, then for some $t$ the moduli space $\overline{\M^{J_t}(\alpha(1),\alpha(0)}$ is nonempty. Here $\partial_\pm$ denotes the differential on the chain complex $ECC(Y_\pm,\lambda_\pm,J_\pm)$. We will use similar notation below without comment.} such that
\[
\partial_- K + K \partial _+ = \phi_0 - \phi_1.
\]
\item{(Composition)}
With the notation of Remark~\ref{rem:composition}, let
\[
\phi_\pm\in\Theta^L(X_\pm,\lambda_\pm,J_\pm)
\]
and
\[
\phi \in \Theta^L(X_-\circ X_+,\lambda_-\circ\lambda_+, J_-\circ J_+).
\]
Then there exists a chain homotopy
\[
K: ECC^L(Y_1,\lambda_1,J_1) \longrightarrow ECC^L(Y_{-1},\lambda_{-1},J_{-1})
\]
such that
\[
\partial_{-1} K + K \partial_1 = \phi_-\circ\phi_+ - \phi
\]
and $K$ counts $J_-\circ_R J_+$-holomorphic currents.
\item{(Trivial Cobordisms)}
Let $\lambda_0$ be a nondegenerate contact form on $Y_0$, and suppose that
\[
(X,\lambda)=\left([a,b]\times Y_0,e^s\lambda_0\right).
\]
Let $J_0$ be a generic $\lambda_0$-compatible almost complex structure on $\R\times Y_0$. Let
\[
\phi_0: ECC^L\left(Y_0,e^b\lambda_0,J_0^{e^b}\right) \longrightarrow ECC^L\left(Y_0,e^a\lambda_0,J_0^{e^a}\right)
\]
denote the chain map \eqref{eqn:scalingisochain}. Let $f:\R\to\R$ be a positive function such that $f(s)=e^{a}$ when $s\le a$ and $f(s)=e^{b}$ when $s\ge b$. Then
\[
\Theta^L\left([a,b]\times Y_0,e^s\lambda_0,J_0^f\right) = \{\phi_0\}.
\]
\end{description}
\end{proposition}

\begin{proof}
In \cite[Thm.\ 1.9]{cc2}, a cobordism map \eqref{eqn:cc2} is constructed, which, given $J$ as in the proposition, is induced by a noncanonical chain map $\phi$ as in \eqref{eqn:cjc}. Such a chain map $\phi$ is defined by counting solutions to the Seiberg-Witten equations on $\overline{X}$, with a Riemannian metric determined by $J$. The Seiberg-Witten equations are perturbed using $r\hat{\omega}$, where $r$ is a large positive constant and $\hat{\omega}$ is a $2$-form on $\overline{X}$ which, up to insignificant factors, agrees with $d\lambda$ on $X$ and with $d\lambda_\pm$ on the ends; see \cite[\S4.2]{cc2}. It is shown in \cite[\S7]{cc2} that the map $\phi$ counts $J$-holomorphic currents if $r$ is sufficiently large. The proof entails showing that given a sequence of solutions to the $r$-perturbed Seiberg-Witten equations with $r\to\infty$, the zero set of one component of the spinor converges to a broken holomorphic current. 

The chain map $\phi$ is not unique, because it depends on $r$, and even for fixed $r$, certain additional perturbations of the Seiberg-Witten equations are needed to obtain transversality of the relevant moduli spaces of solutions, see \cite[\S4.2]{cc2}. Two different perturbations may give rise to different chain maps. In this case one can choose a homotopy between the two perturbations, and the two chain maps will then differ by a chain homotopy which counts solutions to the Seiberg-Witten equations for perturbations in the homotopy; see \cite[Prop.\ 5.2(b)]{cc2}. It is shown in the proof of the latter proposition in \cite[\S7.6]{cc2}, similarly to the proof that the chain maps count holomorphic currents, that if $r$ is sufficiently large then this chain homotopy will count $J$-holomorphic currents. 

We now define $\Phi^L(X,\lambda,J)$ to be the set of chain maps produced by the construction in \cite{cc2} for $r\ge r_0$, where $r_0$ is chosen sufficiently large (depending on $L$, $X$, $\lambda$, and $J$) so that any such chain map counts $J$-holomorphic currents, and any two such chain maps differ by a chain homotopy that counts $J$-holomorphic currents. This definition might depend on $r_0$, but in any case we can now show that as long as $r_0$ is chosen with the above properties, then the sets $\Phi^L(X,\lambda,J)$ will satisfy the requirements of the proposition.

It follows from the above discussion that $\Phi^L(X,\lambda,J)$ is nonempty and satisfies the Holomorphic Curves property. The proof of \cite[Prop.\ 5.2(b)]{cc2} as discussed above also implies the Homotopy Invariance property.

The Composition property likewise follows from the proof of \cite[Prop.\ 5.4]{cc2} in \cite[\S7.6]{cc2}.

Finally, to prove the Trivial Cobordisms property, we need to show that if
\[
\phi\in \Theta^L\left( [a,b]\times Y_0, e^s\lambda_0, J_0^f\right)
\]
then $\phi=\phi_0$. That is, if $\alpha$ and $\beta$ are orbit sets for $\lambda_0$ (which are equivalent to orbit sets for $e^a\lambda_0$ and $e^b\lambda_0$) in which no hyperbolic orbits have multiplicity greater than one, we need to show that $\langle\phi\alpha,\beta\rangle = 1$ if and only if $\alpha=\beta$.

To do so, first note that $J_0^{e^a}$-, $J_0^{e^n}$-, and $J_0^f$-holomorphic curves in $\R\times Y_0$ are equivalent to $J_0$-holomorphic curves in $\R\times Y_0$, via diffeomorphisms that stretch the $\R$ coordinate. Now any $J_0$-holomorphic current in $\R\times Y_0$ that is not a union of covers of trivial cylinders must strictly decrease symplectic action, cf.\ \S\ref{sec:filtered}. Consequently, the only broken $J_0^f$-holomorphic current from $\alpha$ to itself has just one level which is $\R\times\alpha$. It then follows from part (ii) of the ``Holomorphic Curves'' axiom in \cite[Thm.\ 1.9]{cc2} that $\langle\phi\alpha,\alpha\rangle=1$.

Suppose now that $\alpha\neq\beta$; we need to show that $\langle\phi\alpha,\beta\rangle=0$. Suppose to get a contradiction that $\langle\phi\alpha,\beta\rangle=1$. By part (i) of the ``Holomorphic curves'' axiom in \cite[Thm.\ 1.9]{cc2}, there exists a broken $J^f$-holomorphic current
\[
\current=(\current_{N_-},\ldots,\current_{N_+}) \in \overline{\M^{J^f}(\alpha,\beta)}.
\]
Now the chain map $\phi$, by the construction in \cite{cc2}, counts solutions to the Seiberg-Witten equations in index zero moduli spaces. By \cite[Thm.\ 5.1]{gradings}, the index of the Seiberg-Witten moduli space corresponds to the ECH index, with the result that
\[
\sum_{k=N_-}^{N_+}I(\current_k) = 0,
\]
where $I(\current_k)$ denotes the ECH index of $\current_k$, which is reviewed in \cite[\S3.4]{bn}.

On the other hand, each current $\current_k$ is equivalent to a $J_0$-holomorphic current in $\R\times Y$. By \cite[Prop.\ 3.7]{bn}, since $J_0$ is generic, each current $\current_k$ has ECH index $I(\current_k)\ge 0$, with equality if and only if $\current_k$ is a union of trivial cylinders with multiplicities (which by the way is not allowed when $k\neq 0$ by the definition of broken holomorphic current). Since $\alpha\neq\beta$, it follows that the total ECH index
\[
\sum_kI(\current_k) > 0.
\]
This is the desired contradiction.
\end{proof}

\section{Functoriality of the knot filtration}
\label{sec:functoriality}

We are now prepared to prove Theorem~\ref{thm:filtration}, which will finish the proof of Theorem~\ref{thm:main}.

\begin{proof}[Proof of Theorem~\ref{thm:filtration}.]
Let $\lambda_0$ and $\lambda_1$ be two contact forms for $\xi$ for which $B$ is an elliptic Reeb orbit with rotation number $\theta_0$, and let $J_i$ be a generic $\lambda_i$-compatible almost complex structure for $i=0,1$. We need to show that there is a canonical map
\begin{equation}
\label{eqn:functorial}
\Phi_{(\lambda_0,J_0),(\lambda_1,J_1)}: ECH_*^{\mc{F}_B\le R}(Y,\lambda_1,J_1) \longrightarrow ECH_*^{\mc{F}_B\le R}(Y,\lambda_0,J_0)
\end{equation}
with the following properties:
\begin{align}
\label{eqn:functorial1}
\Phi_{(\lambda,J),(\lambda,J)} &= \op{id},\\
\label{eqn:functorial2}  \Phi_{(\lambda_0,J_0),(\lambda_1,J_1)}\circ \Phi_{(\lambda_1,J_1),(\lambda_2,J_2)} &= \Phi_{(\lambda_0,J_0),(\lambda_2,J_2)}.
\end{align}

To define the map \eqref{eqn:functorial}, note that since $\Ker(\lambda_0)=\Ker(\lambda_1)$, we have $\lambda_1=e^{g_1}\lambda_0$ where $g_1:Y\to\R$ is a smooth function.
By the scaling isomorphism \eqref{eqn:scalingiso}, which preserves the filtration $\mc{F}_B$, we may assume that $g_1>0$ everywhere. Now let $g:\R\times Y\to\R$ be a smooth function with the following properties:
\begin{itemize}
\item $g(s,y)=s$ for $s\in(-\infty,\epsilon)$ for some $\epsilon>0$.
\item $g(s,y) = g_1(y) + s - 1$ for $s\in(1-\epsilon,\infty)$ for some $\epsilon>0$.
\item $\partial_sg>0$.
\end{itemize}
Observe that
\[
d(e^g\lambda_0) = e^g\left(d\lambda_0 + (\partial_sg)ds\wedge\lambda_0\right).
\]
is symplectic.
Thus $([0,1]\times Y,e^g\lambda_0)$ is an exact symplectic cobordism from $(Y,\lambda_1)$ to $(Y,\lambda_0)$, and its completion is identified with $\R\times Y$ in the obvious way.

Let $J$ be a cobordism-admissible almost complex structure on $\R\times Y$ which agrees with $J_1$ on $[1,\infty)\times Y$ and with $J_0$ on $(-\infty,0]\times Y$. Since $[1,\infty)\times B$ is a $J_1$-holomorphic submanifold, $(-\infty,0]\times B$ is a $J_0$-holomorphic submanifold, and $[0,1]\times B$ is a symplectic submanifold of $\R\times Y$, we can choose $J$ so that $\R\times B$ is $J$-holomorphic. Moreover, the space of $J$ satisfying the above conditions is contractible.

Given $L>0$, let
\[
\phi: ECC_*^L(Y,\lambda_1,J_1) \longrightarrow ECC_*^L(Y_0,\lambda_0,J_0)
\]
be a chain map in the set $\Phi^L([0,1]\times Y,e^g\lambda_0,J)$ provided by Proposition~\ref{prop:elaboration}. We claim that $\phi$ preserves the filtration $\mc{F}_B$, that is if $\langle\phi\alpha(1),\alpha(0)\rangle\neq 0$ then $\mc{F}_B(\alpha_1)\ge \mc{F}_B(\alpha_0)$. To see this, we know from the Holomorphic Curves property in Proposition~\ref{prop:elaboration} that if $\langle\phi\alpha(1),\alpha(0)\rangle\neq 0$ then there exists a broken holomorphic current
\[
\current=(\current_{N_-},\ldots,\current_{N_+}) \in \overline{\M^J(\alpha(1),\alpha(0))}.
\]
By Lemma~\ref{lem:filtration}, if $k\neq 0$ then $\mc{C}_k$ preserves the filtration $\mc{F}_B$, i.e.\ $\mc{F}_B$ of the orbit set corresponding to the positive ends of $\mc{C}_k$ is greater than or equal to $\mc{F}_B$ of the orbit set corresponding to the negative ends. Since $\R\times B$ is $J$-holomorphic, the same intersection positivity argument shows that $\mc{C}_0$ also preserves the filtration $\mc{F}_B$, and therefore $\mc{C}$ preserves the filtration as well.

Consequently, restricting to the subcomplex where $\mc{F}_B\le R$ and passing to homology, we obtain a map
\begin{equation}
\label{eqn:phistar}
\phi_*: ECH_*^{\stackrel{\A<L}{\mc{F}_B\le R}}(Y_1,\lambda_1,J_1) \longrightarrow ECH_*^{\stackrel{\A<L}{\mc{F}_B\le R}}(Y_0,\lambda_0,J_0).
\end{equation}
Here the superscripts indicate that we are taking the homology of the subcomplex spanned by generators for which both $\mc{A}<L$ and $\mc{F}_B\le R$.

By the Homotopy Invariance property in Proposition~\ref{prop:elaboration} and intersection positivity as in in Lemma~\ref{lem:filtration} again, the map \eqref{eqn:phistar} does not depend on the choice of $g$, $J$, or $\phi\in\Theta^L([0,1]\times Y,e^g\lambda_0,J)$. We now define the map \eqref{eqn:functorial} to be the direct limit of the maps \eqref{eqn:phistar} as $L\to\infty$.

The desired property \eqref{eqn:functorial1} now follows from the Trivial Cobordisms property in Proposition~\ref{prop:elaboration}. And the desired property \eqref{eqn:functorial2} follows from the Composition property in Proposition~\ref{prop:elaboration} together with the same intersection positivity argument as before.
\end{proof}

\begin{remark}
If $B_+$ and $B_-$ are transverse knots in $(Y,\xi)$, define a {\em symplectic cobordism\/} from $B_+$ to $B_-$ to be a compact symplectic submanifold $\Sigma\in[-R,R]\times Y$ for some $R>0$ such that 
\[
\begin{split}
\Sigma\cap ((R-\epsilon,R]\times Y) &= (R-\epsilon,R]\times B_+,\\
\Sigma\cap ([-R,-R+\epsilon)\times Y) &= [-R,-R+\epsilon)\times B_-
\end{split}
\]
for some $\epsilon>0$. The proof of Theorem~\ref{thm:filtration} shows more generally that if $H_1(Y)=0$, then $ECH_*^{\mc{F}\le R}(Y,\xi,\cdot,\theta_0)$ is a functor on the category whose objects are transverse knots in $(Y,\xi)$ and whose morphisms are symplectic cobordisms modulo isotopy. We do not know much about this functor in general, but it could be an interesting topic for future research.
\end{remark}

\begin{remark}
One can also define an analogous functor for any other kind of contact homology in three dimensions whose differential counts holomorphic currents and which admits cobordism maps satisfying an analogue of Proposition~\ref{prop:elaboration}.
\end{remark}

\end{document}